\documentclass[12pt]{article}
\usepackage{amssymb,amsmath}
\usepackage{cases}
\usepackage{amsthm}
\usepackage{amsfonts}
\usepackage{graphicx}
\usepackage{cite,color,xcolor}
\usepackage[left=2.3cm,right=2.3cm,top=2.3cm,bottom=2.3cm]{geometry}
\usepackage[colorlinks,citecolor=blue,urlcolor=blue]{hyperref}
\usepackage[utf8]{inputenc}

\newtheorem{theorem}{Theorem}[section]
\newtheorem{corollary}{Corollary}[section]
\newtheorem{lemma}{Lemma}[section]

\newtheorem{remark}{Remark}[section]

\usepackage{txfonts}

\newcommand{\bal}{\begin{align}}
\newcommand{\bbal}{\begin{align*}}
\newcommand{\beq}{\begin{equation}}
\newcommand{\eeq}{\end{equation}}
\newcommand{\bca}{\begin{cases}}
\newcommand{\eca}{\end{cases}}

\newcommand{\pa}{\partial}
\newcommand{\fr}{\frac}

\newcommand{\De}{\Delta}

\newcommand{\cd}{\cdot}

\newcommand{\dd}{\mathrm{d}}

\newcommand{\R}{\mathbb{R}}

\newcommand\f{\left}
\newcommand\g{\right}
\begin{document}
\bibliographystyle{plain}
\title{Ill-posedness of the Novikov equation in the critical Besov space $B^{1}_{\infty,1}(\R)$}

\author{Jinlu Li$^{1}$, Yanghai Yu$^{2,}$\footnote{E-mail: lijinlu@gnnu.edu.cn; yuyanghai214@sina.com(Corresponding author); mathzwp2010@163.com} and Weipeng Zhu$^{3}$\\
\small $^1$ School of Mathematics and Computer Sciences, Gannan Normal University, Ganzhou 341000, China\\
\small $^2$ School of Mathematics and Statistics, Anhui Normal University, Wuhu 241002, China\\
\small $^3$ School of Mathematics and Big Data, Foshan University, Foshan, Guangdong 528000, China}

\date{\today}

\maketitle\noindent{\hrulefill}

{\bf Abstract:} It is shown that both the Camassa-Holm and Novikov equations are ill-posed in $B_{p,r}^{1+1/p}(\mathbb{R})$ with $(p,r)\in[1,\infty]\times(1,\infty]$ in \cite{Guo2019} and well-posed in $B_{p,1}^{1+1/p}(\mathbb{R})$ with $p\in[1,\infty)$ in \cite{Ye}. Recently, the ill-posedness for the Camassa-Holm equation in $B^{1}_{\infty,1}(\R)$ has been proved in \cite{Guo}. In this paper, we shall solve the only left an endpoint case $r=1$ for the Novikov equation. More precisely, we prove the ill-posedness for the Novikov equation in $B^{1}_{\infty,1}(\R)$ by exhibiting the norm inflation phenomena.

{\bf Keywords:} Novikov equation; Ill-posedness; Critical Besov space.

{\bf MSC (2010):} 35Q53, 37K10.
\vskip0mm\noindent{\hrulefill}

\section{Introduction}
Vladimir Novikov \cite{VN} investigated the question of integrability for Camassa-Holm type equations of the form
$$
(1-\partial_x^2) u_t=P\left(u, u_x, u_{x x}, u_{x x x}, \cdots\right),
$$
where $P$ is a polynomial of $u$ and its $x$-derivatives. Using as test for integrability the existence of an infinite hierarchy of (quasi-) local higher symmetries, he produced about 20 integrable equations with quadratic nonlinearities that include the Camassa-Holm $(\mathrm{CH})$ equation
\bal\label{ch}
(1-\partial_x^2) u_t=u u_{x x x}+2 u_x u_{x x}-3 u u_x
\end{align}
and the Degasperis-Procesi $(\mathrm{DP})$ equation
\bal\label{dp}
(1-\partial_x^2)  u_t=u u_{x x x}+3 u_x u_{x x}-4 u u_x .
\end{align}
Moreover, he produced about 10 integrable equations with cubic nonlinearities that include the following new equation (we may call it the Novikov equation (NE))
\bal\label{ne}
(1-\partial_x^2)  u_t=u^2 u_{x x x}+3 u u_x u_{x x}-4 u^2 u_x.
\end{align}

The Camassa-Holm equation was originally derived as a bi-Hamiltonian system by Fokas and Fuchssteiner \cite{Fokas1981} in the context of the KdV model and gained prominence after Camassa-Holm \cite{Camassa1993} independently re-derived
it from the Euler equations of hydrodynamics using asymptotic expansions. \eqref{ch} is completely integrable \cite{Camassa1993,Constantin-P} with a bi-Hamiltonian structure \cite{Constantin-E,Fokas1981} and infinitely many conservation laws \cite{Camassa1993,Fokas1981}. Also, it admits exact peaked soliton solutions (peakons) of the form $ce^{-|x-ct|}$ with $c>0$, which are orbitally stable \cite{Constantin.Strauss} and models wave breaking (i.e., the solution remains bounded, while its slope becomes unbounded in finite time \cite{Constantin,Escher2,Escher3}).
The Degasperis-Procesi equation with a bi-Hamiltonian structure is integrable \cite{DP} and has traveling wave solutions \cite{Lenells}. Although DP is similar to CH  in several aspects, these two equations are truly
different. One of the novel features of  DP different from  CH is that it has not only peakon solutions \cite{DP} and periodic peakon solutions \cite{YinJFA}, but also shock peakons \cite{Lundmark2007} and the periodic shock waves \cite{Escher}.

For the Novikov equation, Hone-Wang \cite{Home2008} derived the Lax pair  which is given by
\bbal
\left(\begin{array}{l}\psi_1 \\ \psi_2 \\ \psi_3\end{array}\right)_x=U(m, \lambda)\left(\begin{array}{l}\psi_1 \\ \psi_2 \\ \psi_3\end{array}\right), \quad\left(\begin{array}{l}\psi_1 \\ \psi_2 \\ \psi_3\end{array}\right)_t=V(m, u, \lambda)\left(\begin{array}{l}\psi_1 \\ \psi_2 \\ \psi_3\end{array}\right),
\end{align*}
where $m=u-u_{x x}$ and the matrices $U$ and $V$ are defined by
\bbal
U(m, \lambda)=\left(\begin{array}{ccc}
0 & \lambda m & 1 \\
0 & 0 & \lambda m \\
1 & 0 & 0
\end{array}\right)\quad\text{and}\quad V(m, u, \lambda)=\left(\begin{array}{lrr}
\frac{1}{3 \lambda^2}-u u_x & \frac{u_x}{\lambda}-\lambda m u^2 & u_x^2 \\
\frac{u}{\lambda} & -\frac{2}{3 \lambda^2} & -\frac{u_x}{\lambda}-\lambda m u^2 \\
-u^2 & \frac{u}{\lambda} & \frac{1}{3 \lambda^2}+u u_x
\end{array}\right) .
\end{align*}
NE possesses peakon traveling wave solutions \cite{HM,HLS,GH}, which on the real line are given by the formula
$
u(x, t)=\pm \sqrt{c} e^{-|x-c t|}
$
where $c>0$ is the wave speed. In fact, NE possesses multi-peakon traveling wave solutions on both the line and the circle. More precisely, on the line the $n$-peakon
$$
u(x, t)=\sum_{j=1}^n p_j(t) e^{-\left|x-q_j(t)\right|}
$$
is a solution to $\mathrm{NE}$ if and only if the positions $\left(q_1, \ldots, q_n\right)$ and the momenta $\left(p_1, \ldots, p_n\right)$ satisfy the following system of $2 n$ differential equations:
$$
\left\{\begin{aligned}
\frac{\dd q_j}{\dd t} &=u^2(q_j), \\
\frac{\dd p_j}{\dd t} &=-u(q_j) u_x(q_j) p_j .
\end{aligned}\right.
$$
Furthermore, it has infinitely many conserved quantities. Like CH, the most important quantity conserved by a solution $u$ to NE is its $H^1$-norm
$
\|u\|_{H^1}^2 =\int_{\mathbb{R}}(u^2+u_x^2) \dd x.
$

The well-posedness of the Camassa-Holm type equations has been widely investigated during the past 20 years.
The local well-posedness for the Cauchy problem of CH \cite{LO,GB,Dan2,L16} and NE \cite{HH,HH0,N,Wu1,Wu2,Yan1,Yan2} in Sobolev and Besov spaces $B_{p, r}^s(\mathbb{R})$ with  $s>\max\{1+1/{p}, 3/{2}\}$ and $(p,r)\in[1,\infty]\times[1,\infty)$ has been established.
In our recent papers \cite{Li22,Li22-jee}, we established the ill-posedness for CH in $B^s_{p,\infty}(\mathbb{R})$  with $p\in[1,\infty]$ by proving the solution map  starting from $u_0$ is discontinuous at $t = 0$ in the metric of $B^s_{p,\infty}(\mathbb{R})$.
Guo-Liu-Molinet-Yin \cite{Guo2019} established the ill-posedness for
the Camassa-Holm type equations in $B_{p,r}^{1+1/p}(\mathbb{R})$ with $(p,r)\in[1,\infty]\times(1,\infty]$ by proving the norm inflation, which implies that $B_{p, 1}^{1+1/p}$ is the critical Besov space for both CH and NE.  Ye-Yin-Guo \cite{Ye} obtained the local well-posedness for the Camassa-Holm type equation in critical Besov spaces $B^{1+1/p}_{p,1}(\R)$ with $p\in[1,\infty)$. We should mention that the well-posedness for DP in $B^{1}_{\infty,1}(\R)$ has been established in our recent paper \cite{LDP}. Very recently, Guo-Ye-Yin \cite{Guo} obtained the ill-posedness for CH in $B^{1}_{\infty,1}(\R)$ by constructing a special initial data which leads to the norm inflation. However, their initial data seems to be invalid when proving the ill-posedness for NE in $B^{1}_{\infty,1}(\R)$. To the best of our knowledge, whether NE is well-posed or not in in $B^{1}_{\infty,1}(\R)$ is still an open problem. We shall present the negative result in this paper.

Setting $\Lambda^{-2}=(1-\pa^2_x)^{-1}$, then $\Lambda^{-2}f=G*f$ where $G(x)=\fr12e^{-|x|}$ is the kernel of the operator $\Lambda^{-2}$. We can transform the Novikov equation into the following transport type equation
\begin{equation}\label{N}
\begin{cases}
u_t+u^2u_x=\mathbf{P}_1(u)+\mathbf{P}_2(u),\\
u(x,t=0)=u_0(x),
\end{cases}
\end{equation}
where
\begin{equation}\label{6}
\mathbf{P}_1(u)=-\frac12\Lambda^{-2}u_x^3\quad\text{and}\quad \mathbf{P}_2(u)=-\pa_x\Lambda^{-2}\left(\frac32uu^2_x+u^3\right).
\end{equation}

We can now state our main result as follows.
\begin{theorem}\label{th1}
For any $n\in \mathbb{Z}^+$ large enough, there exist $u_0$ and $T>0$ such that the Novikov equation has a solution $u\in \mathcal{C}([0,T);H^{\infty})$ satisfying
\bbal
\|u_0\|_{B^1_{\infty,1}}\leq \frac{1}{\log\log n}\quad\text{but}\quad
\|u(t_0)\|_{B^1_{\infty,1}}\geq {\log\log n}\quad\text{with}\quad t_0\in \left(0,\frac{1}{\log n}\right].
\end{align*}
\end{theorem}
Since the norm inflation implies discontinuous of the data-to-solution at the trivial function $u_0\equiv0$, Theorem \ref{th1} demonstrates that
\begin{corollary}\label{co}
The Cauchy problem for the Novikov equation is ill-posed in $B^1_{\infty,1}(\R)$ in the sense of Hadamard.
\end{corollary}
This paper is structured as follows. In Section \ref{sec2}, we list some notations and known results and recall some Lemmas which will be used in the sequel. In Section \ref{sec3} we present the proof of Theorem \ref{th1} by dividing it into several parts:
(1) Construction of initial data;
(2) Key Estimation for Discontinuity;
(3) The Equation Along the Flow;
(4) Norm inflation.

\section{Preliminaries}\label{sec2}
\quad{\bf Notation}\;
 $C$ stands for some positive constant independent of $n$, which may vary from line to line.
 The symbol $A\approx B$ means that $C^{-1}B\leq A\leq CB$.
 We shall call a ball $B(x_0,r)=\{x\in \R: |x-x_0|\leq R\}$ with $R>0$ and an annulus $\mathcal{C}(0,r_1,r_2)=\{x\in \R: 0<r_1\leq|x|\leq r_2\}$ with $0 <r_1 <r_2$.
  Given a Banach space $X$, we denote its norm by $\|\cdot\|_{X}$. We shall use the simplified notation $\|f,\cdots,g\|_X=\|f\|_X+\cdots+\|g\|_X$ if there is no ambiguity.
 We will also define the Lipschitz space $C^{0,1}$ using the norm $\|f\|_{C^{0,1}}=\|f\|_{L^\infty}+\|\pa_xf\|_{L^\infty}$.
 For $I\subset\R$, we denote by $\mathcal{C}(I;X)$ the set of continuous functions on $I$ with values in $X$. Sometimes we will denote $L^p(0,T;X)$ by $L_T^pX$.
 Let us recall that for all $f\in \mathcal{S}'$, the Fourier transform $\widehat{f}$, is defined by
$$
(\mathcal{F} f)(\xi)=\widehat{f}(\xi)=\int_{\R}e^{-ix\xi}f(x)\dd x \quad\text{for any}\; \xi\in\R.
$$
 The inverse Fourier transform of any $g$ is given by
$$
(\mathcal{F}^{-1} g)(x)=\check{g}(x)=\frac{1}{2 \pi} \int_{\R} g(\xi) e^{i x \cdot \xi} \dd \xi.
$$

Next, we will recall some facts about the Littlewood-Paley decomposition and the nonhomogeneous Besov spaces (see \cite{B} for more details).
Let $\mathcal{B}:=\{\xi\in\mathbb{R}:|\xi|\leq 4/3\}$ and $\mathcal{C}:=\{\xi\in\mathbb{R}:3/4\leq|\xi|\leq 8/3\}.$
Choose a radial, non-negative, smooth function $\chi:\R\mapsto [0,1]$ such that it is supported in $\mathcal{B}$ and $\chi\equiv1$ for $|\xi|\leq3/4$. Setting $\varphi(\xi):=\chi(\xi/2)-\chi(\xi)$, then we deduce that $\varphi$ is supported in $\mathcal{C}$. Moreover,
\begin{eqnarray*}
\chi(\xi)+\sum_{j\geq0}\varphi(2^{-j}\xi)=1 \quad \mbox{ for any } \xi\in \R.
\end{eqnarray*}
We should emphasize that the fact $\varphi(\xi)\equiv 1$ for $4/3\leq |\xi|\leq 3/2$ will be used in the sequel.

For every $u\in \mathcal{S'}(\mathbb{R})$, the inhomogeneous dyadic blocks ${\Delta}_j$ are defined as follows
\begin{numcases}{\Delta_ju=}
0, & if $j\leq-2$;\nonumber\\
\chi(D)u=\mathcal{F}^{-1}(\chi \mathcal{F}u), & if $j=-1$;\nonumber\\
\varphi(2^{-j}D)u=\mathcal{F}^{-1}\f(\varphi(2^{-j}\cdot)\mathcal{F}u\g), & if $j\geq0$.\nonumber
\end{numcases}

Let $s\in\mathbb{R}$ and $(p,r)\in[1, \infty]^2$. The nonhomogeneous Besov space $B^{s}_{p,r}(\R)$ is defined by
\begin{align*}
B^{s}_{p,r}(\R):=\f\{f\in \mathcal{S}'(\R):\;\|f\|_{B^{s}_{p,r}(\mathbb{R})}:=\f\|2^{j s}\| \Delta_j u\|_{L^p}\g\|_{\ell^r(j\geq-1)}<\infty\g\}.
\end{align*}

The following Bernstein's inequalities will be used in the sequel.
\begin{lemma}[Lemma 2.1 in \cite{B}] \label{lem2.1} Let $\mathcal{B}$ be a Ball and $\mathcal{C}$ be an annulus. There exist constants $C>0$ such that for all $k\in \mathbb{N}\cup \{0\}$, any $\lambda\in \R^+$ and any function $f\in L^p$ with $1\leq p \leq q \leq \infty$, we have
\begin{align*}
&{\rm{supp}}\widehat{f}\subset \lambda \mathcal{B}\;\Rightarrow\; \|\pa_x^kf\|_{L^q}\leq C^{k+1}\lambda^{k+(\frac{1}{p}-\frac{1}{q})}\|f\|_{L^p},  \\
&{\rm{supp}}\widehat{f}\subset \lambda \mathcal{C}\;\Rightarrow\; C^{-k-1}\lambda^k\|f\|_{L^p} \leq \|\pa_x^kf\|_{L^p} \leq C^{k+1}\lambda^k\|f\|_{L^p}.
\end{align*}
\end{lemma}
\begin{lemma}[Lemma 2.100 in \cite{B}] \label{lem2.2}
Let $1 \leq r \leq \infty$, $1 \leq p \leq p_{1} \leq \infty$ and $\frac{1}{p_{2}}=\frac{1}{p}-\frac{1}{p_{1}}$. There exists a constant $C$ depending continuously on $p,p_1$, such that
$$
\left\|\left(2^{j}\left\|[\Delta_{j},v \pa_x] f\right\|_{L^{p}}\right)_{j}\right\|_{\ell^{r}} \leq C\left(\|\pa_x v\|_{L^{\infty}}\|f\|_{B_{p, r}^{1}}+\|\pa_x f\|_{L^{p_{2}}}\|\pa_x v\|_{B_{p_1,r}^{0}}\right).
$$
\end{lemma}
\section{Proof of Theorem \ref{th1}}\label{sec3}
In this section, we prove Theorem \ref{th1}.
\subsection{Construction of Initial Data}\label{sec3.1}
Define a smooth cut-off function $\chi$ with values in $[0,1]$ which satisfies
\bbal
\chi(\xi)=
\bca
1, \quad \mathrm{if} \ |\xi|\leq \frac14,\\
0, \quad \mathrm{if} \ |\xi|\geq \frac12.
\eca
\end{align*}
From now on, we set $\gamma:=\frac{17}{24}$ just for the sake of simplicity. Letting
\bbal
n\in 16\mathbb{N}=\left\{16,32,48,\cdots\right\}\quad\text{and}\quad
\mathbb{N}(n)=\left\{k\in 8\mathbb{N}: \frac{n}4 \leq k\leq \frac{n}2\right\}.
\end{align*}
  We introduce the following new notation which will be used often throughout this paper
  \bbal
  \|f\|_{B^k_{\infty,1}\left(\mathbb{N}(n)\right)}=\sum_{j\in\mathbb{N}(n)}2^k\|\Delta_jf\|_{L^\infty},\quad k\in\{0,1\}.
  \end{align*}
Now, we can define the initial data $u_0$ by
\bbal
u_{0}&=n^{-\frac{1}{3}}\f(u^{\rm{H}}_{0}+u^{\rm{L}}_{0}\g),
\end{align*}
where
\bal
u^{\rm{H}}_{0}&:=2^{-n}\log n\sum_{\ell\in \mathbb{N}(n)}\cos\left(2^n\gamma(x+2^{\ell+1}\gamma)\right)\cd
\cos\left(2^{\ell}\gamma(x+2^{\ell+1}\gamma)\right)\cd\check{\chi}(x+2^{\ell+1}\gamma),\label{H}\\
u^{\rm{L}}_{0}&:=\sum_{\ell\in \mathbb{N}(n)}\check{\chi}(x+2^{\ell+1}\gamma).\label{L}
\end{align}
{\bf Some Observations}
\begin{enumerate}
  \item Obviously,
\bal\label{o2}
&\mathrm{supp} \ \widehat{u^{\rm{L}}_{0}}\subset \left\{\xi\in\R: \ |\xi|\leq \fr12\right\}.
\end{align}
  \item It is not difficult to check that
\bal\label{o0}
&\mathrm{supp} \ \mathcal{F}\f(\cos\left(2^{\ell}\gamma(x+2^{\ell+1}\gamma)\right)\cd\check{\chi}(x+2^{\ell+1}\gamma)\g)
\subset \left\{\xi\in\R: \ 2^{\ell}\gamma-\fr12\leq |\xi|\leq 2^{\ell}\gamma+\fr12\right\},
\end{align}
which gives in turn
\bal\label{o00}
&\mathrm{supp} \ \mathcal{F}\f(\cos\left(2^n\gamma(x+2^{\ell+1}\gamma)\right)\cd\cos\left(2^{\ell}\gamma(x+2^{\ell+1}\gamma)\right)
\cd\check{\chi}(x+2^{\ell+1}\gamma)\g)\nonumber\\
&\qquad\subset \left\{\xi\in\R: \ 2^{n}\gamma-2^{\ell}\gamma-\fr12\leq |\xi|\leq 2^{n}\gamma+2^{\ell}\gamma+\fr12\right\}.
\end{align}
Thus
\bal\label{o1}
&\mathrm{supp} \ \widehat{u^{\rm{H}}_{0}}\subset \left\{\xi\in\R: \ \fr43 2^{n-1}\leq |\xi|\leq \fr32 2^{n-1}\right\}.
\end{align}
  \item Since $\check{\chi}$ is a Schwartz function, we have
\bal\label{o3}
|\check{\chi}(x)|+|\pa_x\check{\chi}(x)|\leq C(1+|x|)^{-M}, \qquad  M\gg 1.
\end{align}
  \item By $\check{\chi}(0)=\frac{1}{2\pi}\int_{\R}\chi(x)\dd x$, we have
\bal\label{o4}
\left\|\cos\big(2^{\ell+1}\gamma (x+2^{\ell+1}\gamma)\big)\check{\chi}^3(x+2^{\ell+1}\gamma)\right\|_{L^\infty}\geq \check{\chi}^3(0)\geq\fr1{128\pi}.
\end{align}
\end{enumerate}

\begin{lemma}\label{le-e1}
There exists a positive constant $C$ independent of $n$ such that
\bbal
&2^{n}\|u^{\rm{H}}_{0}\|_{L^\infty}+\|\pa_xu^{\rm{H}}_{0}\|_{L^\infty}\leq C \log n,\\
&\|u^{\rm{L}}_{0}\|_{C^{0,1}}\leq C,\\
&\|u_0\|_{B^1_{\infty,1}}\leq C n^{-\frac13}\log n.
\end{align*}
\end{lemma}
\begin{proof}
Due to \eqref{o1}-\eqref{o3}, by Bernstein's inequality, we have
\bbal
2^{n}\|u^{\rm{H}}_{0}\|_{L^\infty}&+\|\pa_xu^{\rm{H}}_{0}\|_{L^\infty}+\log n\|u^{\rm{L}}_{0}\|_{C^{0,1}}
\\&\leq C\log n\left\|\sum_{\ell\in \mathbb{N}(n)}\frac{1}{(1+|x+2^{\ell+1}\gamma|)^M}\right\|_{L^\infty}\\
&\leq C\log n
\end{align*}
and
\bbal
\|u_0\|_{B^1_{\infty,1}}&\leq C n^{-\frac13}\f(\|u^{\rm{H}}_{0}\|_{B^1_{\infty,1}}+\|u^{\rm{L}}_{0}\|_{B^1_{\infty,1}}\g)\\
&\leq C n^{-\frac13}\f(2^n\|u^{\rm{H}}_{0}\|_{L^{\infty}}+\|u^{\rm{L}}_{0}\|_{L^{\infty}}\g)\\
&\leq C n^{-\frac13}\log n.
\end{align*}
This completes the proof of Lemma \ref{le-e1}.
\end{proof}
\subsection{Key Estimation for Discontinuity}
The following Lemma is crucial for the proof of Theorem \ref{th1}.
\begin{lemma}\label{le-e2}
There exists a positive constant $c$ independent of $n$ such that
\bbal
\left\|u_0(\pa_xu_0)^2\right\|_{B^0_{\infty,1}\left(\mathbb{N}(n)\right)}\geq c(\log n)^2, \qquad n\gg1.
\end{align*}
\end{lemma}
\begin{proof}
Obviously,
\bbal
n\cd u_0(\pa_xu_0)^2&=\underbrace{u^{\rm{L}}_{0}(\pa_xu^{\rm{H}}_{0})^2}_{=:\;\mathbf{I}_1}
+\underbrace{u^{\rm{H}}_{0}(\pa_xu^{\rm{H}}_{0}+\pa_xu^{\rm{L}}_{0})^2}_{=:\;\mathbf{I}_2}
+\underbrace{u^{\rm{L}}_{0}\f((\pa_xu^{\rm{L}}_{0})^2+2\pa_xu^{\rm{H}}_{0}\pa_xu^{\rm{L}}_{0}\g)}_{=:\;\mathbf{I}_3}.
\end{align*}
Next, we need to estimate the above three terms.

{\bf Estimation of $\mathbf{I}_2$.} Using Lemma \ref{le-e1} yields
\bbal
&\left\|\mathbf{I}_2\right\|_{B^0_{\infty,1}\left(\mathbb{N}(n)\right)}\leq C n \|\mathbf{I}_2\|_{L^\infty}\leq n\|u^{\rm{H}}_{0}\|_{L^\infty}\|\pa_xu^{\rm{H}}_{0},\pa_xu^{\rm{L}}_{0}\|^2_{L^{\infty}}\leq Cn2^{-n}(\log n)^2.
\end{align*}

{\bf Estimation of $\mathbf{I}_3$.}
Notice that the support conditions \eqref{o1} and \eqref{o2}, one has
\bbal
&\De_j\mathbf{I}_3=0\quad\text{for}\; j\in \mathbb{N}(n)\quad\Rightarrow\quad
\left\|\mathbf{I}_3\right\|_{B^0_{\infty,1}\left(\mathbb{N}(n)\right)}=0.
\end{align*}

{\bf Estimation of $\mathbf{I}_1$.} Now we focus on the estimation of $\mathbf{I}_1$.
Obviously,
\bbal
\pa_xu^{\rm{H}}_{0}&=-\gamma\log n\sum_{\ell\in \mathbb{N}(n)}\sin\left(2^n\gamma(x+2^{\ell+1}\gamma)\right)\cd
\cos\left(2^{\ell}\gamma(x+2^{\ell+1}\gamma)\right)\cd\check{\chi}(x+2^{\ell+1}\gamma)\\
&\quad+2^{-n}\log n\sum_{\ell\in \mathbb{N}(n)}\cos\left(2^n\gamma(x+2^{\ell+1}\gamma)\right)\cd\pa_x\Big(
\cos\left(2^{\ell}\gamma(x+2^{\ell+1}\gamma)\right)\cd\check{\chi}(x+2^{\ell+1}\gamma)\Big).
\end{align*}
We decompose $\mathbf{I}_1$ into three terms
\bbal
\mathbf{I}_{1}=(\log n)^2\f(\mathbf{I}_{11}+\mathbf{I}_{12}-\mathbf{I}_{13}\g),
\end{align*}
where
\bbal
\mathbf{I}_{11}&=\gamma^2u^{\rm{L}}_{0}\f(\sum_{\ell\in \mathbb{N}(n)}\sin\left(2^n\gamma(x+2^{\ell+1}\gamma)\right)\cd
\cos\left(2^{\ell}\gamma(x+2^{\ell+1}\gamma)\right)\cd\check{\chi}(x+2^{\ell+1}\gamma)\g)^2,
\\ \mathbf{I}_{12}&=2^{-2n}u^{\rm{L}}_{0}\f(\sum_{\ell\in \mathbb{N}(n)}\cos\left(2^n\gamma(x+2^{\ell+1}\gamma)\right)\cd\pa_x\Big(
\cos\left(2^{\ell}\gamma(x+2^{\ell+1}\gamma)\right)\cd\check{\chi}(x+2^{\ell+1}\gamma)\Big)\g)^2,
\\ \mathbf{I}_{13}&=2\gamma2^{-n}u^{\rm{L}}_{0}\sum_{\ell\in \mathbb{N}(n)}\sin\left(2^n\gamma(x+2^{\ell+1}\gamma)\right)\cd
\cos\left(2^{\ell}\gamma(x+2^{\ell+1}\gamma)\right)\cd\check{\chi}(x+2^{\ell+1}\gamma)\\
&\quad\times\sum_{\ell\in \mathbb{N}(n)}\cos\left(2^n\gamma(x+2^{\ell+1}\gamma)\right)\cd\pa_x\Big(
\cos\left(2^{\ell}\gamma(x+2^{\ell+1}\gamma)\right)\cd\check{\chi}(x+2^{\ell+1}\gamma)\Big).
\end{align*}
Easy computations give that
\bbal
\|\mathbf{I}_{12}\|_{L^\infty}&\leq C2^{-2n}\|u^{\rm{L}}_{0}\|_{L^\infty}\left\|\sum_{\ell\in \mathbb{N}(n)}\pa_x\Big(
\cos\left(2^{\ell}\gamma(x+2^{\ell+1}\gamma)\right)\cd\check{\chi}(x+2^{\ell+1}\gamma)\Big)\right\|^2_{L^\infty}
\\&\leq C2^{-2n}\left\|\sum_{\ell\in \mathbb{N}(n)}\frac{2^{\ell}}{(1+|x+2^{\ell+1}\gamma|)^M}\right\|^2_{L^\infty}\\
&\leq C2^{-n}.
\end{align*}
Similarly, we have
$
\|\mathbf{I}_{13}\|_{L^\infty}\leq C 2^{-\frac{n}2}.
$
Thus
\bbal
&\left\|\mathbf{I}_{12},\mathbf{I}_{13}\right\|_{B^0_{\infty,1}\left(\mathbb{N}(n)\right)}\leq C n 2^{-\frac{n}2}.
\end{align*}
By the simple equality $\sin^2(a)\cos^2(b)=\fr14(1-\cos(2a))(1+\cos(2b))$, we break $\mathbf{I}_{11}$ down into some easy-to-handle terms
\bbal
\mathbf{I}_{11}=\gamma^2\sum_{i=1}^5\mathbf{I}_{11i},\quad\text{where}
\end{align*}
\bbal
\mathbf{I}_{111}&=\frac14u^{\rm{L}}_{0}\sum_{\ell\in \mathbb{N}(n)}\cos\big(2^{\ell+1}\gamma (x+2^{\ell+1}\gamma)\big)\check{\chi}^2(x+2^{\ell+1}\gamma) ,\\
\mathbf{I}_{112}&= \frac14u^{\rm{L}}_{0}\sum_{\ell\in \mathbb{N}(n)}\check{\chi}^2(x+2^{\ell+1}\gamma),
\\
\mathbf{I}_{113}&=-\frac14u^{\rm{L}}_{0}\sum_{\ell\in \mathbb{N}(n)}\cos\big(2^{n+1}\gamma (x+2^{\ell+1}\gamma)\big)\cd\check{\chi}^2(x+2^{\ell+1}\gamma),
\\
\mathbf{I}_{114}&= -\frac14u^{\rm{L}}_{0}\sum_{\ell\in \mathbb{N}(n)}\cos\big(2^{n+1}\gamma (x+2^{\ell+1}\gamma)\big)\cd\cos\big(2^{\ell+1}\gamma (x+2^{\ell+1}\gamma)\big)\cd\check{\chi}^2(x+2^{\ell+1}\gamma),
\\
\mathbf{I}_{115}&=u^{\rm{L}}_{0}\sum_{\ell,j\in \mathbb{N}(n)\atop\ell\neq j}\Big(\sin\left(2^n\gamma(x+2^{\ell+1}\gamma)\right)\cd
\cos\left(2^{\ell}\gamma(x+2^{\ell+1}\gamma)\right)\cd\check{\chi}(x+2^{\ell+1}\gamma)
  \\& \quad  \times \sin\left(2^n\gamma(x+2^{j+1}\gamma)\right)\cd
\cos\left(2^{j}\gamma(x+2^{j+1}\gamma)\right)\cd\check{\chi}(x+2^{j+1}\gamma)\Big).
\end{align*}
Notice that the support conditions \eqref{o0} and \eqref{o00}, one has
\bbal
&\De_j\mathbf{I}_{112}=\De_j\mathbf{I}_{113}=\De_j\mathbf{I}_{114}=0\quad\text{for}\; j\in \mathbb{N}(n),
\end{align*}
which implies directly that
\bbal
&\left\|\mathbf{I}_{112},\mathbf{I}_{113},\mathbf{I}_{114}\right\|_{B^0_{\infty,1}\left(\mathbb{N}(n)\right)}=0.
\end{align*}
Using Lemma \ref{le-e1}, we have
\bbal
\|\mathbf{I}_{115}\|_{B^0_{\infty,1}(\mathbb{N}(n))}&\leq Cn\|\mathbf{I}_{115}\|_{L^\infty}
\leq Cn\|u^{\rm{L}}_{0}\|_{L^\infty}\sum_{\ell,j\in \mathbb{N}(n)\atop
\ell\neq j}\left\|\check{\chi}(x+2^{\ell+1}\gamma)
  \cd\check{\chi}(x+2^{j+1}\gamma)\right\|_{L^\infty}
\\&\leq Cn\sum_{j>\ell\in \mathbb{N}(n)}\left\|(1+|x+2^{j+1}\gamma|^2)^{-M}(1+|x+2^{\ell+1}\gamma|^2)^{-M}\right\|_{L^\infty}\\
&\leq Cn\sum_{j>\ell\in \mathbb{N}(n)}\left\|(1+|x|^2)^{-M}(1+|x-(2^{j+1}-2^{\ell+1})\gamma|^2))^{-M}\right\|_{L^\infty}\\&\leq Cn\sum_{j>\ell\in \mathbb{N}(n)}\left(\gamma (2^{j}-2^{\ell})\right)^{-2M}
\\&\leq Cn2^{-\frac{nM}{2}},
\end{align*}
where we have separated $\R$ into two different regions $\{x: |x|\leq\gamma (2^{j}-2^{\ell})\}$ and $\{x: |x| > \gamma (2^{j}-2^{\ell})\}$.

Finally, we can break $\mathbf{I}_{111}$ down into two parts, where the first part  contributes the main part.
\bbal
\mathbf{I}_{111}&=\sum_{\ell\in \mathbb{N}(n)}\cos\big(2^{\ell+1}\gamma (x+2^{\ell+1}\gamma)\big)\check{\chi}^3(x+2^{\ell+1}\gamma)\\
&\quad+\sum_{\ell,j\in \mathbb{N}(n)\atop
\ell\neq j}\cos\big(2^{\ell+1}\gamma (x+2^{\ell+1}\gamma)\big)\check{\chi}^2(x+2^{\ell+1}\gamma)\check{\chi}(x+2^{j+1}\gamma)\\
&:=\mathbf{I}_{1111}+\mathbf{I}_{1112}.
\end{align*}
Due to \eqref{o0} and the support condition of $\varphi$ and for all $k\in \mathbb{Z}$
\bbal
\varphi(2^{-k}\xi)\equiv 1\quad \text{for}\quad \xi\in  \mathcal{C}_k\equiv\left\{\xi\in\R^d: \ \frac{4}{3}2^{k}\leq |\xi|\leq \frac{3}{2}2^{k}\right\},
\end{align*}
 we have
\begin{equation*}
{\dot{\Delta}_j\mathbf{I}_{1111}=\mathcal{F}^{-1}\left(\varphi(2^{-j}\cdot)\mathcal{F}\mathbf{I}_{1111}\right)=}
\begin{cases}
\cos\big(2^{j+1}\gamma (x+2^{j+1}\gamma)\big)\check{\chi}^3(x+2^{j+1}\gamma), &\text{if}\; \ell=j,\\
0, &\text{if}\; \ell\neq j,
\end{cases}
\end{equation*}
which combining \eqref{o4} implies that
\bbal
\|\mathbf{I}_{1111}\|_{B^0_{\infty,1}(\mathbb{N}(n))}&=\sum_{j\in \mathbb{N}(n)}\left\|\cos\big(2^{j+1}\gamma (x+2^{j+1}\gamma)\big)\check{\chi}^3(x+2^{j+1}\gamma)\right\|_{L^\infty}\geq cn.
\end{align*}
Following the same procedure as $\mathbf{I}_{115}$, we get
\bbal
\|\mathbf{I}_{1112}\|_{B^0_{\infty,1}(\mathbb{N}(n))}\leq  Cn2^{-\frac{nM}{2}},
\end{align*}
Gathering all the above estimates, we obtain that for large enough $n$
\bbal
n\|u_0(\pa_xu_0)^2\|_{B^0_{\infty,1}(\mathbb{N}(n))}&\geq \|\mathbf{I}_{1}\|_{B^0_{\infty,1}(\mathbb{N}(n))}-\|\mathbf{I}_{2}\|_{B^0_{\infty,1}(\mathbb{N}(n))}\\
&\geq C(\log n)^2\f(\|\mathbf{I}_{1111}\|_{B^0_{\infty,1}(\mathbb{N}(n))}-
\|\mathbf{I}_{12},\mathbf{I}_{13},\mathbf{I}_{115},\mathbf{I}_{1112}\|_{B^0_{\infty,1}(\mathbb{N}(n))}-n2^{-n}\g)\\
&\geq cn(\log n)^2.
\end{align*}
This completes the proof of Lemma \ref{le-e2}.
\end{proof}
\begin{remark}\label{dis}
Setting $u_0=n^{-\fr12}u^{\mathrm{H}}_0$ where $u^{\mathrm{H}}_0$ is given by \eqref{H} and following the same argument as the proof of Lemma \ref{le-e2}, we can establish
\bbal
\left\|u_0\right\|_{B^1_{\infty,1}(\R)}\leq C_1n^{-\fr12}\log n\quad\text{and}\quad\left\|(\pa_xu_0)^2\right\|_{B^0_{\infty,1}\left(\mathbb{N}(n)\right)}\geq c_2(\log n)^2, \qquad n\gg1.
\end{align*}
\end{remark}

\subsection{The Equation Along the Flow}\label{sec3.2}

Given a Lipschitz velocity field $u$, we may solve the following ODE to find the flow induced by $u^2$:
\begin{align}\label{ode}
\quad\begin{cases}
\frac{\dd}{\dd t}\phi(t,x)=u^2(t,\phi(t,x)),\\
\phi(0,x)=x,
\end{cases}
\end{align}
which is equivalent to the integral form
\bal\label{n}
\phi(t,x)=x+\int^t_0u^2(\tau,\phi(\tau,x))\dd \tau.
\end{align}
Considering
\begin{align}\label{pde}
\quad\begin{cases}
\pa_tv+u^2\pa_xv=P,\\
v(0,x)=v_0(x),
\end{cases}
\end{align}
then, we get from \eqref{pde} that
\bbal
\pa_t(\De_jv)+u^2\pa_x\De_jv&=R_j+\Delta_jP,
\end{align*}
with $R_j=[u^2,\De_j]\pa_xv=u^2\De_j\pa_xv-\Delta_j(u^2\pa_xv)$.

Due to \eqref{ode}, then
\bbal
\frac{\dd}{\dd t}\left((\De_jv)\circ\phi\right)&=R_j\circ\phi+\Delta_jP\circ\phi,
\end{align*}
which means that
\bal\label{l6}
\De_jv\circ\phi=\De_jv_0+\int^t_0R_j\circ\phi\dd \tau+\int^t_0\Delta_jP\circ\phi\dd \tau.
\end{align}
\subsection{Norm Inflation}\label{sec3.3}

For $n\gg1$, we have for $t\in[0,1]$
\bbal
\|u\|_{C^{0,1}}\leq C\|u_0\|_{C^{0,1}}\leq C n^{-\frac{1}{3}}\log  n.
\end{align*}
To prove Theorem \ref{th1}, it suffices to show that there exists $t_0\in(0,\frac{1}{\log n}]$ such that
\bal\label{m}
\|u(t_0,\cdot)\|_{B^1_{\infty,1}}\geq \log\log n.
\end{align}
We prove \eqref{m} by contraction. If \eqref{m} were not true, then
\bal\label{m1}
\sup_{t\in(0,\frac{1}{\log n}]}\|u(t,\cdot)\|_{B^1_{\infty,1}}< \log\log n.
\end{align}
We divide the proof into two steps.

{\bf Step 1: Lower bounds for $(\De_ju)\circ \phi$}

Now we consider the equation along the Lagrangian flow-map associated to $u^2$.
Utilizing \eqref{l6} to \eqref{N} yields
\bbal
(\De_ju)\circ \phi&=\De_ju_0+\int^t_0R^1_j\circ \phi\dd \tau +\int^t_0\De_jF\circ \phi\dd \tau
\\&\quad +\int^t_0\big(\De_jE\circ \phi-\De_jE_0\big)\dd \tau+t\De_jE_0,
\end{align*}
where
\bbal
&R^1_j=[u^2,\De_j]\pa_xu,\qquad
F=-\Lambda^{-2}\f(\fr12(\pa_xu)^3+\pa_x(u^3)\g), \\
&E=-\frac32\pa_x\Lambda^{-2}\f(u(\pa_xu)^2\g)\quad\text{with}\quad E_0=-\frac32\pa_x\Lambda^{-2}\f(u_0(\pa_xu_0)^2\g).
\end{align*}
Due to Lemma \ref{le-e2}, we deduce
\bal\label{g1}
\sum_{j\in \mathbb{N}(n)}2^j\|\De_jE_0\|_{L^\infty}
&\approx \sum_{j\in \mathbb{N}(n)}\|\De_j\pa_xE_0\|_{L^\infty}\nonumber\\
&\geq C\sum_{j\in \mathbb{N}(n)}\f\|\De_j[u_0(\pa_xu_0)^2]\g\|_{L^\infty}\nonumber\\
&\geq c(\log n)^2.
\end{align}
Notice that \eqref{n}, then we have for $t\in(0,\frac{1}{\log n}]$,
\bbal
\frac12\leq |\pa_x\phi|\leq 2,
\end{align*}
thus
\bbal
\|f(t,\phi(t,x))\|_{L^\infty}= \|f(t,\cdot)\|_{L^\infty}.
\end{align*}
Then, using the commutator estimate from Lemma \ref{lem2.2}, we have
\bal\label{g2}
\sum_{j\geq -1}2^j\|R^1_j\circ \phi\|_{L^\infty}&\leq C\sum_{j\geq -1}2^j\|R^1_j\|_{L^\infty}\nonumber\\
&\leq C\|\pa_x(u^2)\|_{B^0_{\infty,1}}\|u\|_{B^1_{\infty,1}}\nonumber
\\
&\leq C\|u\|_{C^{0,1}}\|u\|^2_{B^1_{\infty,1}}\nonumber\\
&\leq C n^{-\frac13}(\log  n)^3.
\end{align}
Also, we have
\bal\label{g3}
\sum_{j\in \mathbb{N}(n)}2^j\|\De_jF\circ \phi\|_{L^\infty}&\leq C\sum_{j\in \mathbb{N}(n)}2^j\|\De_jF\|_{L^\infty}
\nonumber\\
&\leq C\|(\pa_xu)^3+\pa_x(u^3)\|_{L^\infty}\nonumber\\
&\leq C\|u\|^3_{C^{0,1}}
\leq C\|u_0\|^3_{C^{0,1}}\leq Cn^{-1}(\log n)^3.
\end{align}
Combining \eqref{g1}-\eqref{g3} and using Lemmas \ref{le-e1}-\ref{le-e2} yields
\bbal
\sum_{j\in \mathbb{N}(n)}2^j\|(\De_ju)\circ \phi\|_{L^\infty}
&\geq t\sum_{j\in \mathbb{N}(n)}2^j\|\De_jE_0\|_{L^\infty}-\sum_{j\in \mathbb{N}(n)}2^j\|\De_jE\circ \phi-\De_jE_0\|_{L^\infty}\\
&\quad-Cn^{-\frac{1}{3}}(\log  n)^3
-C\|u_0\|_{B^1_{\infty,1}}
\\&\geq ct\log^2n-\sum_{j\in \mathbb{N}(n)}2^j\|\De_jE\circ \phi-\De_jE_0\|_{L^\infty}-Cn^{-\frac{1}{3}}(\log  n)^3.
\end{align*}
{\bf Step 2: Upper bounds for $\De_jE\circ \phi-\De_jE_0$}

By easy computations,
$$\pa_x\Lambda^{-2}[u^2u_x^3+2u\pa_x(u^2u_x)u_x]
=\pa_x\Lambda^{-2}[2u^2u_x^3+\pa_x(u^3u_x^2)]=2\pa_x\Lambda^{-2}(u^2u_x^3)+\Lambda^{-2}(u^3u_x^2)-u^3u_x^2,$$
then we find that
\bal\label{E}
\pa_tE+u^2\pa_xE&=-\fr32\f(\pa_x\Lambda^{-2}\pa_t(uu_x^2)+u^2\pa_x^2\Lambda^{-2}(uu_x^2)\g)\nonumber\\
&=\mathrm{J}+\fr32\f(\pa_x\Lambda^{-2}[u^2u_x^3+2u\pa_x(u^2u_x)u_x]+u^3u_x^2-u^2\Lambda^{-2}(uu_x^2)\g)\nonumber\\
&=\mathrm{J}+\mathrm{K},
\end{align}
where
\bbal
&\mathrm{J}=-\frac32\pa_x\Lambda^{-2}\f(\mathbf{P}_1(u)u^2_x+\mathbf{P}_2(u)u^2_x+2\pa_x\mathbf{P}_1(u)uu_x+2\pa_x\mathbf{P}_2(u)uu_x\g),\\
&\mathrm{K}=\fr32\f(2\pa_x\Lambda^{-2}(u^2u_x^3)+\Lambda^{-2}(u^3u_x^2)-u^2\Lambda^{-2}(uu_x^2)\g).
\end{align*}
Utilizing \eqref{l6} to \eqref{E} yields
\bbal
\De_jE\circ \phi-\De_jE_0=\int^t_0[u^2,\De_j]\pa_xE\circ \phi\dd \tau +\int^t_0\De_j(\mathrm{J}+\mathrm{K})\circ \phi\dd \tau.
\end{align*}
Using the commutator estimate from Lemma \ref{lem2.2}, one has
\bbal
2^j\|[u^2,\De_j]\pa_xE\|_{L^\infty}&\leq C(\|\pa_x(u^2)\|_{L^\infty}\|E\|_{B^1_{\infty,\infty}}
+\|\pa_xE\|_{L^\infty}\|u\|_{B^1_{\infty,\infty}})\\
&\leq C\|u\|^5_{C^{0,1}}\leq Cn^{-\frac53}(\log n)^5.
\end{align*}
Due to the facts
\bbal
\|\Lambda^{-2}f\|_{L^\infty}\leq \|f\|_{L^\infty}\quad\text{and}\quad \|\pa_x\Lambda^{-2}f\|_{L^\infty}\leq \|f\|_{L^\infty}\quad \Rightarrow\quad \|\pa^2_x\Lambda^{-2}f\|_{L^\infty}\leq 2\|f\|_{L^\infty},
\end{align*}
then we have
\bbal
2^j\|\De_j\mathrm{J}\|_{L^\infty}\approx\|\pa_x\mathrm{J}\|_{L^\infty}\leq C \|u\|^5_{C^{0,1}}\leq Cn^{-\frac53}(\log n)^5.
\end{align*}
Similarly,
\bbal
2^j\|\De_j\mathrm{K}\|_{L^\infty}\leq C \|u\|^5_{C^{0,1}}\leq Cn^{-\frac53}(\log n)^5.
\end{align*}
Then, we deduce that
\bbal
2^j\|\De_jE\circ \phi-\De_jE_0\|_{L^\infty}\leq C\|u\|^5_{C^{0,1}}\leq Cn^{-\frac53}(\log n)^5,
\end{align*}
which leads to
\bbal
\sum_{j\in \mathbb{N}(n)}2^j\|\De_jE\circ \phi-\De_jE_0\|_{L^\infty}
\leq Cn^{-\frac23}(\log n)^5.
\end{align*}
Combining Step 1 and Step 2, then for $t=\frac{1}{\log n}$, we obtain for $n\gg1$
\bbal
\|u(t)\|_{B^1_{\infty,1}}&\geq \|u(t)\|_{B^1_{\infty,1}(\mathbb{N}(n))}\\
&\geq C\sum_{j\in \mathbb{N}(n)}2^j\|(\De_ju)\circ \phi\|_{L^\infty}
\\&\geq c t(\log n)^2-Cn^{-\frac23}(\log n)^5-Cn^{-\frac{1}{3}}(\log  n)^3\\
&\geq \log\log n,
\end{align*}
which contradicts the hypothesis \eqref{m1}. Thus, Theorem \ref{th1} is proved.{\hfill $\square$}
\section{Discussion}
\quad  By the clever Lagrangian coordinate transformation used in \cite{Guo} and constructing a new initial data, we prove that the Novikov equation is ill-posed in critical Besov spaces $B^{1}_{\infty,1}(\R)$. Thus our results (Theorem \ref{th1} and Corollary \ref{co}) indicate that the local well-posedness and ill-posedness for the Novikov equation in all
critical Besov spaces $B^{1+1/p}_{p,r}(\R)$ have been solved completely. Since the Novikov equation has cubic nonlinear term, we expect that
norm inflation is stemmed from the worst term $u(\pa_xu)^2$, which is different from the quadratic term $(\pa_xu)^2$ for the Camassa-Holm equation. Our new idea is to construct a initial data which includes two parts, one of whose Fourier transform is supported at high frequencies and the other is supported at low frequencies. Then the cubic nonlinear term $u(\pa_xu)^2$ will generate the low-high-high frequency interaction, which contributes a large quantity lead to the norm inflation. Lastly, we should mention that, by dropping the low frequency term, the initial data $u_0=2^{-\fr{n}2}u_0^{\mathrm{H}}$ can be as an example which leads to the norm inflation for the Camassa-Holm  equation (see Remark \ref{dis}) .
\section*{Acknowledgments}
J. Li is supported by the National Natural Science Foundation of China (11801090 and 12161004) and Jiangxi Provincial Natural Science Foundation (20212BAB211004). Y. Yu is supported by the National Natural Science Foundation of China (12101011) and Natural Science Foundation of Anhui Province (1908085QA05). W. Zhu is supported by the National Natural Science Foundation of China (12201118) and Guangdong
Basic and Applied Basic Research Foundation (2021A1515111018).

\section*{Data Availability} No data was used for the research described in the article.

\section*{Conflict of interest}
The authors declare that they have no conflict of interest.

\end{document}